\documentclass[11pt]{amsart}
\usepackage{amssymb}

\title[elliptic space curve]{Galois lines for normal elliptic space curves, II}
\author[Hisao Yoshihara]{}

\newtheorem{theorem}{Theorem}
\newtheorem{lemma}{Lemma}
\newtheorem{claim}{Claim}

\newtheorem{corollary}[lemma]{Corollary}

\theoremstyle{definition}
\newtheorem{definition}{Definition}

\theoremstyle{remark}
\newtheorem{remark}[lemma]{Remark}

\newenvironment{namelist}[1]{%
\begin{list}{}
  {
   \settowidth{\labelwidth}{#1}
   \setlength{\leftmargin}{2.5\labelwidth}}
}{%
\end{list}}

\begin{document}
\maketitle

\begin{center}

{\sc Hisao Yoshihara}\\
\medskip
{\small{\em Department of Mathematics, Faculty of Science, Niigata University,\\
Niigata 950-2181, Japan}\\
E-mail:{\tt yosihara@math.sc.niigata-u.ac.jp}}
\end{center}

\begin{abstract}
For each linearly normal elliptic curve $C$ in $\mathbb P^3$,  
we determine Galois lines and their arrangement. 
The results are as follows: the curve $C$ has just six $V_4$-lines and in case $j(C)=1$, 
it has eight $Z_4$-lines in addition. 
The $V_4$-lines form the edges of a tetrahedron, in case $j(C)=1$, 
for each vertex of the tetrahedron, there exist just two $Z_4$-lines passing through it. 
We obtain as a corollary that each plane quartic curve of genus one does not have more than one Galois point. 
\medskip

\noindent key words and phrases : Galois line, space elliptic curve, Galois covering  

\noindent 2000 Mathematics Subject Classification number : 14H50, 14H52
\end{abstract}
\bigskip

\section{Introduction} 
This is a continuation of \cite{cy}, where we found three $V_4$-lines for each linearly normal elliptic curve $C$ in $\mathbb P^3$, 
and four $Z_4$-lines for such curve $C$ with $j(C)=1$. 
However, those lines are not all the ones. 
In this article we determine all Galois lines and describe their arrangement. 
First let us recall the definition of Galois lines briefly.  

Let $k$ be the ground field of our discussion, we assume it to be algebraically closed, later we assume it 
the field  $\mathbb C$ of complex numbers.  
Let $C$ be a smooth irreducible non-degenerate curve of degree $d$ in the projective three space ${\mathbb P}^3$ and $\ell$ 
a line in ${\mathbb P}^3$ not meeting $C$. Let $\pi_{\ell} : {\mathbb P}^3 \dashrightarrow l_0$ be the projection with center $\ell$, 
where ${\ell}_0$ is a line not meeting $\ell$. Restricting  
$\pi_{\ell}$ to $C$, we get a surjective morphism $\pi_{\ell}|_C : C \longrightarrow l_0$ and hence an extension of fields
 $(\pi_{\ell}|_C)^* : k({\ell}_0) \hookrightarrow k(C)$, where $[k(C):k(\ell_0)]=d$. 
 Note that the extension of fields does not depend on $\ell_0$, but on $\ell$.

\begin{definition}\label{1}
The line $\ell$ is said to be a Galois line for $C$ if the extension $k(C)/k(\ell_0)$ is Galois, or equivalently, if $\pi_{\ell}|_C$ 
is a Galois covering.  
In this case Gal($k(C)/k({\ell}_0)$) is said to be the Galois group for $\ell$ and denoted by $G_{\ell}$.

\end{definition}

If $\ell$ is the Galois line, then each element $\sigma \in G_l$ induces an automorphism of $C$ over ${\ell}_0$. We denote it by the same letter $\sigma$. 
Hereafter, assume $C$ is linearly normal, i.e., the hyperplanes cut out the complete linear series $|\mathcal{O}_C(1)|$. 
Then, the automorphism $\sigma$ can be extended to a projective transformation of $\mathbb P^3$, 
which will be also denoted by the same letter $\sigma$. 

\bigskip 

We use the following notation and convention:

\begin{namelist}{$\cdot$}
\item[{$\cdot$}]$V_4$ : the Klein 4-group
\item[{$\cdot$}]$Z_4$ : the cyclic group of order four 
\item[{$\cdot$}]$\sim$ : the linear equivalence of divisors
\item[{$\cdot$}]Aut$(C)$ : the automorphism group of $C$ 
\item[{$\cdot$}]${\mathcal L}(D):= \{\ f \in k(C) \setminus \{ 0 \} \ | \ \mathrm{div}(f)+D \ge 0 \}\cup \{0 \}$, where div$(f)$ is the divisor of $f$ and
 $D$ is a divisor on $C$.
\item[{$\cdot$}]$\langle \cdots \rangle$ : the group generated by the set $\{ \cdots \}$ or the linear subvariety spanned by the set $\{ \cdots \}$ 
\item[{$\cdot$}]$V(F)$ : the variety defined by $F=0$
\item[{$\cdot$}]$C \cdot H$ : the intersection divisor of $C$ and $H$ on $C$, where $H$ is a plane.   
\item[{$\cdot$}]$\ell_{PQ}$ : the line passing through $P$ and $Q$ 
\end{namelist}

\section{Statement of Results}
We assume $k=\mathbb C$ and use the same notation as in \cite{cy}. 

\begin{definition}\label{3}
When $\ell$ is a Galois line for $C$ and $G_{\ell} \cong V_4$ (resp. $Z_4$), we call $\ell$ a $V_4$ (resp. $Z_4$)-line.
\end{definition}

There exist $V_4$-lines for the curve which is given by an intersection of hypersurfaces as follows. 

\begin{lemma}\label{5}
Suppose $S_1$ and $S_2$ are irreducible quadratic surfaces in $\mathbb P^3$ satisfying the following conditions{\rm :}  
\begin{enumerate}
\item[(1)] $S_i$ $(i=1, 2)$ has a singular point $Q_i$ and $Q_1 \ne Q_2$.
\item[(2)] $S_1 \cap S_2$ is a smooth curve $\Delta$. 
\item[(3)] The line $\ell$ passing through $Q_1$ and $Q_2$ does not meet $\Delta$.
\end{enumerate}
Then, $\Delta$ is a linearly normal elliptic curve and $\ell$ is a $V_4$-line for $\Delta$. 
\end{lemma}

Let $C$ be a linearly normal elliptic curve in $\mathbb P^3$. 
Then, there exists a divisor $D$ of degree four on an elliptic curve $E$ such that $C$ is given by an embedding of $E$ associated with the complete linear system $|D|$.  
Note that $C$ can be expressed as an intersection of two quadratic surfaces. 

\begin{lemma}\label{13}
There exist just four irreducible quadratic surfaces $S_i$ {\rm (} $0 \le i \le 3$ {\rm )} such that each $S_i$ has a singular point and contains $C$. 
Let $Q_i$ be the unique singular point of $S_i$. 
Then the four points are not coplanar. 
\end{lemma}

\begin{remark}\label{23}
Let $\pi_Q : \mathbb P^3 \dasharrow \mathbb P^2$ be the projection with center $Q \in \mathbb P^3 \setminus C$. 
If $\pi_Q$ induces a $2$ to $1$ morphism from $C$ onto its image in Lemma \ref{13}, then $Q$ coincides with one of $Q_i$. 
\end{remark}

The main theorem is stated as follows:

\begin{theorem}\label{6}
For each linearly normal elliptic curve in $\mathbb P^3$, there exist four non-coplanar points $Q_i$ {\rm (}$0 \le i \le 3${\rm )}  
such that the lines passing through each two of them are $V_4$-lines for $C$. 
Namely, all the $V_4$-lines form the six edges of a tetrahedron. 
Further, if the Weierstrass normal form of $E$ is given by $y^2=4(x-e_1)(x-e_2)(x-e_3)$, then we can present explicitly the coordinates of $Q_i$ {\rm (}by taking a suitable coordinates 
of $\mathbb P^3${\rm )}  
as follows: 
\[
  Q_0=(0:0:0:1) \ \mathrm{and} \  Q_{i}=(1:-c_i:e_i:0), \ (i=1,2,3) , 
\]
where $c_i={e_i}^2+e_je_k$ such that $\{ i, j, k \}=\{1, 2, 3 \}$. 

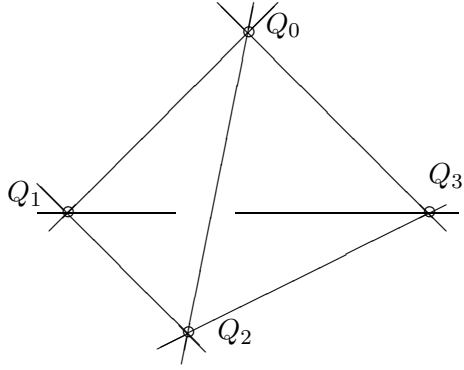
\begin{center}
\setlength{\unitlength}{0.8mm}
\begin{picture}(120,70)
\put(49,4){$\circ$}
\put(89,24){$\circ$}
\put(29,24){$\circ$}
\put(59,54){$\circ$}
\put(63,55){$Q_0$}
\put(20,27){$Q_1$}
\put(55,4){$Q_2$}
\put(90,30){$Q_3$}
\put(60,55){\line(-1,-1){33}}
\put(60,55){\line(1,1){5}}
\put(60,55){\line(-1,-5){11}}
\put(60,55){\line(1,5){1}}
\put(60,55){\line(1,-1){33}}
\put(60,55){\line(-1,1){5}}
\put(30,25){\line(1,-1){23}}
\put(30,25){\line(-1,1){5}}
\put(50,5){\line(2,1){43}}
\put(50,5){\line(-2,-1){5}}
\put(30,25){\line(1,0){18}}
\put(30,25){\line(-1,0){5}}
\put(90,25){\line(-1,0){32}}
\put(90,25){\line(1,0){5}}
\end{picture}
\end{center}
\end{theorem}

\begin{remark}\label{7}
In the case of an elliptic curve $E$ in $\mathbb P^2$ it has a Galois point 
if and only if $j(E)=0$, and then it has just three $Z_3$-points. 
\end{remark}

In the case where the $j$-invariant $j(C)=1$, there exists an automorphism of order four with a fixed point. 
This curve has the other Galois lines as follows. 

\begin{theorem}\label{8}
Under the same assumption as in Theorem \ref{6}, if $j(C)=1$, then there exist eight 
$Z_4$-lines {\rm (}in addition to the $V_4$-lines{\rm )}. 
To state in more detail, for each vertex $Q_i$ {\rm (}$0 \le i \le 3${\rm )} of the tetrahedron in Theorem \ref{6}, there exist two $Z_4$-lines passing through it.  
Therefore, for each vertex, there exist three $V_4$-lines and two $Z_4$-lines passing through it  
and the total number of Galois lines is fourteen. 
Two $Z_4$-lines do not meet except at one of the vertices.  
\end{theorem}

Let $\Sigma$ be the set of six $V_4$-lines in Theorem \ref{6}. In the case where $j(C)=1$ let $\Sigma'$ 
be the set of eight $Z_4$-lines in Theorem \ref{8}. 
The following corollary is an answer to the question for the case of outer Galois point \cite[Theorem 2]{m}. 
 
\begin{corollary}\label{18} 
For a plane quartic curve $\Gamma$ with genus one, the number of {\rm (}outer{\rm )} Galois points 
is at most one. If $\Gamma$ has the Galois point, then the Galois group $G$ is isomorphic to $V_4$ or $Z_4$. 
Further, if $G \cong V_4$ {\rm (}resp. $Z_4${\rm )}, then $\Gamma$ is obtained by a projection $\pi_Q : \mathbb P^3 \dasharrow \mathbb P^2$ with 
center $Q$, where $Q \in \Sigma$ {\rm (}resp. $Q \in \Sigma'${\rm )} such that $Q \ne Q_i$ {\rm (}$0 \le i \le 3${\rm )}. 
\end{corollary}

\begin{remark}\label{22}
Different from the case of the space quartic curve, a plane quartic curve of genus one   
does not necessarily have a Galois point. 
\end{remark}

\begin{remark}\label{4}
Since $C$ is given by the embedding associated with a complete linear system 
and has a Galois line, the embedding is called a Galois embedding, which has been defined in \cite{y3}.   
\end{remark}

\section{Proofs}
First we prove Lemma \ref{5}. 
It is easy to see that $\Delta$ has genus one and \\ $\dim \operatorname{H}^0(\Delta, \ \mathcal O_{\Delta}(1))=4$. Hence $\Delta$ is a linearly normal elliptic curve.  
Let $\pi_{Q_i}$ be the projection $\mathbb P^3 \dasharrow P^2$ with center $Q_i$ ($i=1,\ 2$) 
and put $\Delta_i=\pi_{Q_i}(\Delta) \subset \mathbb P^2$ and $R_i=\pi_{Q_i}(\ell \setminus \{ Q_i \})$. 
Then $\Delta_i$ is a conic and $R_i$ is a point not on $\Delta_i$. 
Let $\varpi_{R_i}$ be the projection $\mathbb P^2 \dasharrow \mathbb P^1$ with center $R_i$. 
Restricting $\varpi_{R_i}$ to $\Delta_i$, we get a surjective morphism $\varpi_{R_i}|_{\Delta_i} : C_i \longrightarrow \mathbb P^1$. 
Therefore we have two morphisms 
\[
\pi_i=\varpi_{Q_i} \cdot \pi_{Q_i} : \Delta \longrightarrow \mathbb P^1 
\] 
of degree four. They coincide with the restriction of the projection 
$\pi_{\ell} : \mathbb P^3 \dasharrow \mathbb P^1$. 
Note that $k(\Delta_1)$ and $k(\Delta_2)$ are distinct subfields of $k(\Delta)$ 
and $[k(\Delta):k(\Delta_i)]=[k(\Delta_i):k(\mathbb P^1)]=2$. 
We infer that $k(\Delta)$ is a $V_4$-extension of $k(\mathbb P^1)$, hence $\pi_{\ell}|_{\Delta}$ is a $V_4$-Galois covering. 
This proves Lemma \ref{5}. 

Fix a universal covering $\pi : \mathbb C \longrightarrow \mathbb C/\mathcal L$, where $\mathcal L$ 
is the lattice in $\mathbb C$ defining a complex torus. 
We assume $\mathcal L=\mathbb Z+\mathbb Z\omega$, where $\Im \omega >0$. 
Let $\wp(z)$ be the Weierstrass $\wp$-function with respect to $\mathcal L$. 
Then, the map $\varphi : \mathbb C \longrightarrow E $ defined by \\ $\varphi(z)=(\wp(z) : \ \wp'(z) : 1)$, 
induces an isomorphism $\bar{\varphi} : \mathbb C/\mathcal L \longrightarrow E$. 
The defining equation of the elliptic curve $E$ is the Weierstrass normal form $y^2=4x^3+px+q$. We assume it to be
 factored as $4(x-e_1)(x-e_2)(x-e_3)$. 
Put 
$P_{\alpha}=\varphi(\alpha)$ for $\alpha \in \mathbb C$.  
Denote by $+$ the sum of divisors on $E$ and, at the same time, the sum of complex numbers. For example, 
$P_{\alpha}+P_{\beta}$ and $\alpha+\beta$ denote the sum of divisors and complex numbers respectively. 

\begin{lemma}\label{10}
We have the linear equivalence of divisors on $E${\rm :}  
\[
P_{\alpha}+P_{\beta} \sim P_{\alpha + \beta} + P_{0}. 
\]  
\end{lemma}

\begin{proof}
This may be well-known. See, for example, \cite[Ch. IV, Theorem 4,13B]{h}.  
\end{proof}

\begin{lemma}\label{9}
Let $D$ be the divisor of degree four on $E$. 
By taking a suitable translation $\tau$ on $E$,  we have $\tau^*(D) \sim 4P_0$. 
\end{lemma} 

\begin{proof}
Suppose $D=\sum_{i=1}^4 P_{\alpha_i}$. Then, take 
$\beta=-\sum_{i=1}^4 \alpha_i/4$. 
Let $\tau$ be the translation on $E$ induced from the one $z \mapsto z+\beta$ on $\mathbb C$. 
Then we have  
$\tau^*(D)=\sum_{i=1}^4 P_{\alpha_i+\beta} $. 
Using Lemma \ref{10}, we get $\tau^*(D) \sim 4P_0$. 
\end{proof}

Let $D$ be a hypeplane section of $C$. Applying Lemma \ref{9}, we see that 
there exists an elliptic curve $C_0$ in $\mathbb P^3$ given by the embedding associated with $|4P_0|$ 
and an isomorphism $\psi : \mathbb P^3 \longrightarrow \mathbb P^3$ satisfying that 
$\psi(C_0)=C$ and $4P_0 \sim {\psi}^*(D)$. 
So that we have the following lemma. 

\begin{lemma}\label{19}
We can assume $C$ is given by the embedding associated with $|4P_0|$.
\end{lemma}

Therefore it is sufficient for our purpose to consider the curve embedded by $|4P_0|$. 
Let $\phi : E \longrightarrow C \subset \mathbb P^3$ be the embedding of $E$ associated with $|4P_0|$. 

\begin{center}
\setlength{\unitlength}{0.8mm}
\begin{picture}(120,70)
\put(20,50){$\mathbb C$}
\put(20,10){$\mathbb C/\mathcal L$}
\put(60,10){$E$}
\put(100,10){$\mathbb P^3$}
\put(15,30){$\pi$}
\put(43,32){$\varphi$}
\put(40,5){$\bar{\varphi}$}
\put(75,5){$\phi$}
\put(86,10){$C$}
\put(93,10){$\subset$}
\put(22,45){\vector(0,-1){28}}
\put(31,11){\vector(1,0){25}}
\put(25,45){\vector(1,-1){30}}
\put(66,11){\vector(1,0){18}}
\end{picture}
\end{center}

In order to study the number and arrangement of Galois lines, we provide some lemmas.  
Let $\mathcal S$ and $\mathcal G$ be the set of Galois lines  for $C$ and the set of subgroups of Aut$(C)$ respectively.   
Since a Galois line $\ell$ determine the Galois group $G_{\ell}$ in Aut$(C)$ uniquely, we can define the following map.

\begin{definition}\label{14}
We define an arrangement-map $\rho : \mathcal S \longrightarrow \mathcal G$ by $\rho(\ell)=G_{\ell}$. 
\end{definition}

We study the map $\rho$ in detail. 
Note that each element of $G_{\ell}$ can be extended to 
a projective transformation. That is, we have a faithful representation $r : G_{\ell} \longrightarrow PGL(3, \mathbb C)$.

\begin{lemma}\label{15}
The map $\rho$ is injective.
\end{lemma}

\begin{proof}
For two elements $\ell_i$ of $\mathcal S \ (i=1, 2)$, suppose $\rho({\ell_1})=\rho({\ell_2})$ and $\ell_1 \ne \ell_2$. 
Then, the following two cases take place:   
\begin{enumerate} 
\item[(i)] $\ell_1 \cap \ell_2$ consists of one point $P$. 
\item[(ii)] $\ell_1 \cap \ell_2 = \emptyset$. 
\end{enumerate}
In the case (i), for a general point $Q \in C$, put $H_{iQ}=\langle \ell_i, Q \rangle$ $(i=1, 2)$ : the plane spanned by $\ell_i$ and $Q$. 
Since $G_{\ell_1}=G_{\ell_2}$, we have $H_{1Q} \cap \ell_0=H_{2Q} \cap \ell_0=\{ R \}$, where $\ell_0$ is the line defined in Introduction.  
Further, since $\pi_{\ell_1}(H_{1Q} \cap C)=\pi_{\ell_2}(H_{2Q} \cap C)= R$, the set of four points 
$H_{1Q} \cap C$ is equal to that of $H_{2Q} \cap C$ and they lie on the line $H_{1Q} \cap H_{2Q}$, which passes through $P$. 
This implies $C$ is contained in the plane spanned by $\ell_0$ and $P$. 
Since $C$ is assumed to be non-degenerate, this is a contradiction. 
Next we treat the case (ii). 
Similarly, for a general point $Q \in C$, put $H_{iQ}=\langle \ell_i, \ Q \rangle$. 
Then, by the same argument as above, the four points $H_{1Q} \cap C$ and $H_{2Q} \cap C$ 
lie on the line $H_{1Q} \cap H_{2Q}$. Thus $C$ is contained in a rational normal scroll  
$\Sigma$. 
However, $H_{iQ} \cap \Sigma$ is a line, so that $\Sigma$ must be a plane. This is a contradiction. 
\end{proof}

We present a criterion when $G \subset {\rm Aut}(C)$ can be the image of an element of $\mathcal S$. See \cite[Theorem 2.2]{y3} for a similar one. 
Hereafter we use the notation ${P_{\alpha}}'=\phi(P_{\alpha})=(\phi\varphi)(\alpha) \in C$ for brevity. 

\begin{lemma}\label{16}
A subgroup $G=\{ \sigma_1, \ldots, \sigma_4 \} $ of Aut$(C)$ is an image of $\rho$ if and only if $G$ satisfies the following condition $(\diamondsuit)${\rm :}  
\begin{enumerate}
\item[$(\diamondsuit)$] For each point $Q \in C$ the divisor $\sum_{i=1}^4 \sigma_i(Q)$ is linearly equivalent to $4{P_0}'$ and $C/G$ is a rational curve.
\end{enumerate}
\end{lemma}

\begin{proof}
If $G=\rho(\ell)$, then clearly $C/G \cong \mathbb P^1$.  Take a plane $H$ satisfying that $H \supset \ell$ and $H \ni Q$. 
By definition the point $\sigma_i(Q)$ $(1 \le i \le 4)$ lies on $H$, hence the divisor is linearly equivalent to $4{P_0}'$. 
Conversely, for a point $Q \in C$, put $D=\sum_{i=1}^4 \sigma_i(Q)$. 
By assumption we have $D \sim 4{P_0}'$, hence $G$ acts on $\operatorname{H}^0(C, \ \mathcal O_{C}(1))$. 
Therefore each element of $G$ can be extended to a projective transformation. 
Letting $\pi : C \longrightarrow C/G \cong \mathbb P^1$, 
we take independent sections $s_0$ and $s_1$ of $\operatorname{H}^0(\mathbb P^1, \ \mathcal O_{\mathbb P^1}(1))$ 
and put $\widetilde{s_i}={\pi}^*(s_i) \ (i=1, 2)$. 
Then we have ${\sigma}^*(\widetilde{s_i})=\widetilde{s_i}$.  Taking a basis of $\operatorname{H}^0(C, \ \mathcal O_{C}(1))$  
containing  $\widetilde{s_1}$ and $\widetilde{s_2}$, we obtain a Galois line $\ell$ such that $\rho(\ell)=G$. 
\end{proof}

We study whether $\ell_1 \cap \ell_2 = \emptyset$ or $\ne \emptyset$ by observing $G_{\ell_1} \cap G_{\ell_2}$ in Aut$(C)$.

\begin{lemma}\label{11} 
Suppose $\ell_1$ and $\ell_2$ are distinct Galois lines. Then, the following two cases take place.  
\begin{enumerate}
\item[(1)] If $\ell_1 \cap \ell_2 = \emptyset$, then $G_{\ell_1} \cap G_{\ell_2} =
\{ \mathrm{id} \}$ in {\rm Aut}$(C)$. 
\item[(2)] If $\ell_1 \cap \ell_2$ is a point $P$, then it is a singular point of some quadratic surface containing $C$, 
Further, we have $G_{\ell_1} \cap G_{\ell_2} = \langle \sigma \rangle$, where $\sigma$ has order two and has a fixed point as an 
automorphism of $C$.
\end{enumerate}   
\end{lemma}

\begin{proof}
Take an element $\sigma \in G_{\ell_1} \cap G_{\ell_2}$. It can be extended to a projective transformation. 
Since every plane $H_i$ containing $\ell_i$ is invariant by $\sigma$, we infer  $\sigma(\ell_i)=\ell_i$ (i=1, 2). Therefore, for each hyperplane $H_1 \supset \ell_1$, if $H_1 \cap \ell_2=
\{ Q \}$, then $\sigma(Q)=Q$, i.e., ${\sigma}|_{\ell_2}=$ id. 
By the same argument we also have $\sigma|_{\ell_1}=$ id. 
Since $\ell_1 \cap \ell_2 = \emptyset$, $\sigma$ is identity on $\mathbb P^3$. 
Next we treat the second case. 
Suppose $\ell_1 \cap \ell_2$ consists of one point $P$. 
Then, for each point $Q \in C$, put $H_{iQ}=\langle \ell_i, Q \rangle$ and $\ell_Q=H_{1Q} \cap H_{2Q}$. Since $H_{iQ} \supset \ell_{Q}$ for $i=1$ and $2$, 
we have $\sigma(Q) \in \ell_Q$. 
Therefore $C$ is contained in the cone passing through $P$. 
Clearly the order of $\sigma$ is two. Since the quotient curve $C/\langle \sigma \rangle$ is isomorphic to 
$\ell_0$, the $\sigma$ has a fixed point in $C$.   
\end{proof}

From Lemma \ref{11} we infer the following remark.

\begin{remark}\label{20}
Let $\ell$ be a  Galois line and take a point $P \in \ell$. 
Let $\pi_P : \mathbb P^3 \dasharrow \mathbb P^2$ be a projection with center $P$. 
If $P$ is not the vertex of the tetrahedron, then $\pi_P(\ell \setminus \{ P \})$ is a Galois point for 
the quartic curve $\pi_P(C)$. However, if $P$ is the one, then 
$\pi_P|_C$ turns out to be a $2$ to $1$ morphism onto its image and $\pi_P(C)$ is a conic in $\mathbb P^2$.   
\end{remark}

Hereafter we denote by 
$\sigma_i$ ($0 \le i \le 3$) an automorphism of $E$ such that the representation on 
$\mathbb C$ is 
\[
\sigma_0 (z)=-z, \ \ \sigma_1 (z)=-z+\frac{1}{2}, \ \ \sigma_2(z)=-z+\frac{\omega}{2}, \ \ \sigma_3(z)=-z+\frac{1+\omega}{2}. 
\]

\begin{lemma}\label{12}
The number of $V_4$-lines is at most six. 
\end{lemma}

\begin{proof}
Suppose $C$ has a $V_4$-line $\ell$. Then, let $H$ be a plane containing $\ell$ and 
${P_0}'$. 
Since $\pi_{\ell}|_C : C \longrightarrow \mathbb P^1$ is a $V_4$-covering, the intersection divisor $H \cdot C$ on $C$ can be expressed in one of the following two types: 
\begin{enumerate}
\item[(i)] $H \cdot C = 2{P_0}' + 2{P_{\gamma}}'$ 
\item[(ii)] $H \cdot C = {P_0}' +{P_{\gamma_1}}'+{P_{\gamma_2}}'+{P_{\gamma_3}}'$. 
\end{enumerate}
Suppose $G=\langle \sigma, \tau \rangle$, where 

\begin{equation}
\sigma(z)=-z+\alpha \ \mathrm{and} \ \tau(z)=z+\beta 
\label{eqn:a}
\end{equation}
on the universal covering $\mathbb C$, where $2\beta \equiv 0 \pmod{\mathcal L}$ and 
$\beta \not\equiv 0 \pmod{\mathcal L}$. 
The case (i) (resp. (ii)) occurs when $\alpha \equiv 0 \pmod{\mathcal L}$ (resp. $\alpha \not\equiv 0 \pmod{\mathcal L}$) in (\ref{eqn:a}). 
We consider the possibility of $\alpha \not\equiv 0$, i.e., we treat the case (ii). 
Since $H \cdot C$ is invariant by the action of $G$, 
it can be expressed as ${P_0}'+{P_{\alpha}}'+{P_{\beta}}'+{P_{\alpha + \beta}}'$. 
Since this is linearly equivalent to $4{P_0}'$, we infer 

\begin{equation}
P_{\alpha}+P_{\beta}+P_{\alpha + \beta} \sim 3P_0 
\label{2}
\end{equation} 
on $E$. 
The left hand side of (\ref{2}) is linearly equivalent to 
$P_{2(\alpha + \beta)}+2P_0$ by Lemma \ref{10}. 
Therefore we have $P_{2(\alpha + \beta)} \sim P_0$. 
This implies $2(\alpha + \beta) \equiv 0 \pmod{\mathcal L}$, i.e., 
$2\alpha \equiv 0 \pmod{\mathcal L}$. 
Then, let us find the distinct subgroups $G$ of Aut$(C)$ such that $G$ is generated by 
order two elements. 
By taking two from $\sigma_i (0 \le i \le 3)$, we have six subgroups $G_{ij}=\langle \sigma_i, \ \sigma_j \rangle$, where 
$0 \le i < j \le 3$. Clearly $G_{ij} \cong V_4$.  
For example, $G_{12}=\{ \mathrm{id}, \ \sigma_1, \ \sigma_2,\ \sigma_1\sigma_2 \}$, where $(\sigma_1\sigma_2)(z)=z+(1+\omega)/2$. 
\end{proof}

\begin{lemma}\label{24}
Putiing $a_i=(e_i-e_j)(e_i-e_k)$, we have 
\[
{\sigma_0}^*(x)=x, \ \ {\sigma_0}^*(y)=-y
\]
and 
\[
{\sigma_i}^*(x)=\frac{a_i}{x-e_i} + e_i, \ \ {\sigma_i}^*(y)=\frac{a_i}{(x-e_i)^2}y, \ where \ 1 \le i \le 3. 
\]
\end{lemma}

\begin{proof}
Since $x=\wp(z)$ and $y=\wp'(z)$, we can prove them by using the the addition formulas of $\wp$ and $\wp'$:   
\[
\begin{array}{ccl}
\wp(z_1+z_2) & = & -\wp(z_1)-\wp(z_2)+\frac{1}{4} \left( \frac{\wp'(z_1)-\wp'(z_2)}{\wp(z_1)-\wp(z_2)} \right)^2 \ \mathrm{and} \\
\wp'(z_1+z_2) & =& 
\frac{-1}{\wp(z_1)-\wp(z_2)} \left[ \wp'(z_1) \left \{ (-\wp(z_1)-2\wp(z_2))+ \frac{1}{4} \left(\frac{\wp'(z_1)-\wp'(z_2)}{\wp(z_1)-\wp(z_2)}  \right)^2   \right \} \right. \\ 
 & & \left. +\wp'(z_2)\left \{ (2\wp(z_1)+\wp(z_2)-\frac{1}{4} \left (\frac{\wp'(z_1)-\wp'(z_2)}{\wp(z_1)-\wp(z_2)} \right )^2 \right \} \right]
\end{array}
\]
\end{proof}

Since $\mathcal L(4P_0)=\langle 1, \ x^2,\  x, \ y \rangle$, we can assume the curve 
$C$ is given by the embedding $\phi(x,y)=(1:x^2:x:y)$. Let $(X:Y:Z:W)$ be a set of homogeneous coordinates on $\mathbb P^3$. Then  
the ideal of $C$ is generated by 
\[
F_1=XY-Z^2 \ \mathrm{and} \ F_2=4YZ+pXZ+qX^2-W^2. 
\]

\begin{lemma}\label{25}
Using the  same notation $G_{ij}=\langle \sigma_i, \ \sigma_j \rangle$ as in the proof of Lemma \ref{12}, we denote by 
$K_{ij}=k(x,y)^{G_{ij}}$ the fixed subfield of $k(x,y)$ by $G_{ij}$. 
Then we have 
\[
K_{0i}=k \left( \frac{x^2+c_i}{x-c_i} \right),\ where \ 1 \le i \le 3
\]
and 
\[
K_{ij}=k \left( \frac{y}{c_k+2e_kx-x^2} \right),\ where \ 1 \le i < j \le 3 \ and \ (k-i)(k-j) \ne 0. 
\]
In particular, the Galois lines which correspond to $G_{0i}$ and $G_{ij}$ by the arrangement-map $\rho$ are 
\[
Y+c_iX=Z-e_iX=0 \ \ and \ \ c_kX-Y+2e_kZ=W=0
\]
respectively.
\end{lemma}

\begin{proof}
By making use of Lemma \ref{24}, we can check the assertions by direct calculations. 
\end{proof}

Now we proceed with the proof of Lemma \ref{13}. 
Let $S=V(F)$ be a surface containing $C$. 
Then $F$ can be expressed as $\lambda_1F_1+\lambda_2F_2$, where $(\lambda_1 : \lambda_2) \in \mathbb P^1$. 
In case $\lambda_2=0$, the point $Q_0=(0:0:0:1)$ is the singular point of $V(F_1)$. 
On the other hand, in case $\lambda_2 \ne 0$, put $b=\lambda_1/\lambda_2$. 
So we assume  $F=bF_1+F_2$. Consider the condition that $V(F)$ has a singular point, i.e., consider the simultaneous linear equations 

\begin{equation}
F_X = F_Y=F_Z=F_W=0. 
 \label{eqn:1}
\end{equation}

This is equivalent to consider the rank of the matrix 
\begin{equation}
M_b = \left( 
\begin{array}{cccc}
2q & b & p & 0 \\
b & 0 & 4 & 0 \\
p & 4 & -2b & 0 \\
0 & 0 & 0 & -2
\end{array}
\right).
\end{equation}
The equations (\ref{eqn:1}) have a non-trivial solution if and only if 
\begin{equation}
  b^3+4pb-16q=0. 
  \label{eqn:2}
\end{equation}
It easy to see that the left hand side of (\ref{eqn:2})~ can be factored into \\ $(b+4e_1)(b+4e_2)(b+4e_3)$. 
Thus, there exist three distinct solutions of (\ref{eqn:1}).
Since the rank of $M_b$ is three for each solution of (\ref{eqn:1}), each surface $S_i=V(b_iF_1+F_2)$ is irreducible, where $b_i=-4e_i$. 
Let $Q_{i}$ be the unique singular point of 
$S_i$. By simple calculations we obtain $Q_{i}=(8:-2p-{b_i}^2:-2b_i:0)=(1:-c_i:e_i:0)$, where  $c_i={e_i}^2+e_je_k$ such that $\{ i, j, k \}=\{1, 2, 3 \}$. 
Since 
\[
\det \left( 
\begin{array}{ccc}
1 & -c_1 & e_1 \\
1 & -c_2 & e_2 \\
1 & -c_3 & e_3
\end{array}
\right)
=2(e_1-e_2)(e_2-e_3)(e_3-e_1) \ne 0, 
\] 
the four points are not coplanar. 
This completes the proof. 
\bigskip

The proof of Remark \ref{23} is as follows. 
Let $\Sigma_Q$ be the set $\{ \ \ell_{QR} \ | \ R \in C \  \}$. 
Then there exists a cone $S_Q$ with the singularity at $Q$ such that $S_Q \supset C$ and $S_Q \supset \Sigma$. 
Therefore, by Lemma \ref{13}, we have $Q=Q_i$ for some $i$.  

\bigskip

Combining Lemmas \ref{5}, \ref{13} and \ref{12}, we infer readily Theorem \ref{6}.  

\begin{remark}\label{21}
By using the condition $(\diamondsuit)$ in Lemma \ref{16}, we can prove that the number of $V_4$-lines is 
just six. However, Lemmas \ref{5} and \ref{13} give the more detailed structure of the arrangement of $V_4$-lines. 
\end{remark}

Now we go to the proof of Theorem \ref{8}. Since $j(C)=1$, we can assume $\omega=\sqrt{-1}$.  
Hereafter, for simplicity we use $i$ instead of $\sqrt{-1}$, so $\mathcal L=\mathbb Z+\mathbb Zi$. 

\begin{lemma}\label{17}
The number of $Z_4$-lines is at most eight. 
\end{lemma}

\begin{proof}
Suppose $C$ has a $Z_4$-line $\ell$. Then, let $H$ be a plane containing $\ell$ and 
${P_0}'$. 
Since $\pi_{\ell}|_C : C \longrightarrow \mathbb P^1$ is a $Z_4$-covering, one of the following three cases take place: 
\begin{enumerate}
\item[(i)]$H \cdot C=4{P_0}'$. 
\item[(ii)] $H \cdot C = 2{P_0}' + 2{P_{\gamma}}'$ 
\item[(iii)] $H \cdot C = {P_0}' +{P_{\gamma_1}}'+{P_{\gamma_2}}'+{P_{\gamma_3}}'$. 
\end{enumerate}
Suppose $G=\langle \sigma \rangle$, where 
\begin{equation}
\sigma(z)=i z+\alpha  
\label{eqn:6}
\end{equation}
 on the universal covering $\mathbb C$. 
The case (i) occurs if and only if ${P_0}'$ is a fixed point for $\sigma$, i.e., $\alpha \equiv 0 \pmod{\mathcal L}$ in (\ref{eqn:6}). 
The case (ii) occurs if and only if ${P_0}'$ is a fixed point for ${\sigma}^2$, i.e., $2\alpha \equiv 0 \pmod{\mathcal L}$ in (\ref{eqn:6}).  
Concerning the last case (iii), since $H \cdot C$ is invariant by the action of $G$, it  
can be expressed as ${P_0}'+{P_{\alpha}}'+{P_{i \alpha}}'+{P_{(1+i)\alpha}}'$. 
Since this is linearly equivalent to $4{P_0}'$, we infer 
\begin{equation}
P_{\alpha}+P_{i \alpha}+P_{(1+i)\alpha} \sim 3P_0 
\label{4}
\end{equation} 
on the curve $E$. 
Moreover the left hand side of (\ref{4}) is linearly equivalent to 
$P_{2(1+i)\alpha} + 2P_0$ by Lemma \ref{10}. 
Therefore we have $P_{2(1+i)}\alpha \sim P_0$. 
This implies $2(1+i)\alpha \equiv 0 \pmod{\mathcal L}$. 
To find the possibility of $\alpha$, it is sufficient to solve the equation  
$2(1+i)\alpha \equiv 0 \pmod{\mathcal L}$. 
By a simple calculation we have $\alpha=(m+ni)/4$, where 
\[
(m,n)=(0,\ 0), \ (2,\ 2), \ (2,\ 0), \ (0,\ 2), \ (3,\ 1), \ (1,\ 3),\ (1,\ 1), \ (3,\ 3). 
\]
Thus we get eight subgroups, which might be the images of $\rho$ of Definition \ref{14}. 
\end{proof}

Checking the condition $(\diamondsuit)$ of Lemma \ref{16}, we now prove Theorem \ref{8}. 
As we see from the proof of Lemma \ref{17}, we have $G=\langle \sigma \rangle$, where $\sigma(z)=iz+\alpha$. 
Since $\sigma$ has fixed points, the curve $C/G$ is rational.  
For each point $Q \in C$ there exists $\gamma \in \mathbb C$ satisfying that $Q={P_{\gamma}'}$. 
So it is sufficient to prove that ${P_{\gamma}}'+ {P_{\sigma(\gamma)}}'+{P_{\sigma^2(\gamma)}}'+{P_{\sigma^3(\gamma)}}'
 \sim 4{P_0}'$.  
 Since $2(1+i)\alpha \equiv 0 \pmod{\mathcal L}$ as in the proof of Lemma \ref{17}, this holds true by 
Lemma \ref{10}. 
Since $j(C)=1$, we can assume $y^2=4x^3-x$ and hence $e_1=1/2, \ e_2=-1/2, \ e_3=0$. 
Thus we have $Q_0=(0:0:0:1), \ Q_1=(4:-1:2:0), \ Q_2=(4:-1:-2:0)$ and $Q_3=(4:1:0:0)$.  
Let $\ell_1$ and $\ell_2$ are $Z_4$-lines and $G_{\ell_1}=\langle \tau_1 \rangle$ and $G_{\ell_2}=\langle \tau_2 \rangle$. 
If $\ell_1$ and $\ell_2$ meet, then we have ${\tau_1}^2={\tau_2}^2$ by Lemma \ref{11}. 
Letting $\tau_1(z)=iz+\alpha_1$ and $\tau_2(z)=iz+\alpha_2$, we have $(1+i)(\alpha_1-\alpha_2) \in \mathcal L$. 
Denote by $\ell(m,n)$ the line corresponding to the group $\langle \tau \rangle$ by the arrangement-map $\rho$, where 
$\tau(z)=iz+(m+ni)/4$. 
The following assertion is easy to see. 

\begin{claim}
Putting $\sigma_{mn}(z)=iz+(m+ni)/4$ and $G_{mn}=\langle \sigma_{mn} \rangle$, 
we have $G_{00} \cap G_{22}=\langle \sigma_0 \rangle, \ G_{20} \cap G_{02}=\langle \sigma_3 \rangle, \ 
G_{11} \cap G_{33}=\langle \sigma_2 \rangle$ and  $G_{31} \cap G_{13}=\langle \sigma_1 \rangle$. 
\end{claim}

\begin{claim}
The intersections of the eight $Z_4$-lines are   
 $\ell(0,0) \cap \ell(2,2)=Q_0$, $\ell(2,0) \cap \ell(0,2)=Q_3$, $\ell(1,1) \cap \ell(3,3)=Q_2$ and $\ell(3,1) \cap \ell(1,3)=Q_1$. 
\end{claim}

\begin{proof}
The intersection points are found by Lemma \ref{25}.  For example, 
the point $\ell(1,1) \cap \ell(3,3)$ is found as follows: 
Since $G_{11} \cap G_{33}=\langle \sigma_2 \rangle$, the point is the intersection of 
two lines 
\[
c_3X-Y+2e_1Z=W=0 \ \mathrm{and} \ c_1X-Y+2e_1=W=0, 
\] 
where $e_1=1/2, \ e_3=0$ and $c_1=1/4, \ c_3=-1/4$. 
So it is $Q_2$.  

\end{proof}

\medskip

Now, we prove Corollary \ref{18}.
Let $E$ be the Weierstrass normal form of the normalization of $\Gamma$ and let $\mu : E \longrightarrow \Gamma \subset \mathbb P^3$ 
be the normalization morphism. 
Put $D=\mu^*(L)$ for a line $L$ in $\mathbb P^2$. 
By Lemma \ref{19} we can assume $C$ is given by the embedding by $|4P_0|$. 
Therefore, $\Gamma$ is regained as $\pi_P(C)$, where $\pi_P : \mathbb P^3 \dasharrow \mathbb P^2$ is the projection with center $P$. 
Suppose $\Gamma$ has two Galois points $Q_1$ and $Q_2$. 
Then, letting $\ell_1={\pi_P}^*(Q_1)$ and $\ell_2={\pi_P}^*(Q_2)$, they are Galois lines for $C$ and $\ell_1 \cap \ell_2=\{ P \}$. 
However, as we have seen Remark \ref{20}, the projection $\pi_P$ induces a $2$ to $1$ morphism from $C$ to 
$\Gamma$ and $\pi_P(C)$ is a rational curve, this is a contradiction. 
On the other hand, if $P$ lies in one of the Galois lines, i.e., $P \in \ell$ and is not the vertex, then 
$\pi_P$ induces a birational transformation on $C$ by Remark \ref{23} and $\pi_P(\ell \setminus \{ P \})$ is a Galois point for 
$\Gamma = \pi_P(C)$. 

Finally, we mention Remark \ref{22}. 
Take a point $Q \in \mathbb P^3$ which does not lie on the Galois lines. 
Then, the curve $\Gamma = \pi_Q(C)$  is a quartic curve with no Galois point. 
Because, by Remark \ref{23} it is birational to $C$. 
Suppose it has a Galois point. Then, there exists a smooth quartic curve $C'$ in $\mathbb P^3$ and 
a Galois line ${\ell}'$ and a point $P' \in \mathbb P^3$ satisfying that $\pi_{P'}(C')=\Gamma$. 
Moreover, there exists an isomorphism $\varphi : \mathbb P^3 \longrightarrow \mathbb P^3$ such that 
$\varphi(C')=C$ and $\varphi({\ell}')$ coincides with some Galois line for $C$. Since ${\ell}' \ni P'$, 
we have $\varphi({\ell}') \ni P$, which is a contradiction.

Thus we complete all proofs.

\bigskip

\noindent {\bf Problem.} 
We  ask the following questions concerning Galois embedding of elliptic curves.
\begin{enumerate}
\item[(a)] In case $\ell$ is not a Galois line, consider the Galois group $G$ of the 
Galois closure curve \cite[Definition 1.3]{y2}. If $\ell$ is general, then the Galois group 
is a full symmetric group \cite[Theorem 2.2]{y2}, see also \cite{ps}. So we ask if $\ell$ is neither general (i.e., $G \not\cong S_4$) nor Galois, then 
what group can appear.  
For the group which appears, how are the arrangements of the lines with the group?  
\item[(b)]Let $D$ be a divisor of degree $d \geq 5$ on $E$. Then, study the Galois embedding by 
$|D|$. In particular, consider the Galois group and the arrangement of Galois subspaces (\cite{y3}).
\end{enumerate}

\bigskip

\bibliographystyle{amsplain}

\end{document}